\RequirePackage{etoolbox}
\csdef{input@path}{{style/}{graphics/}}
\documentclass[aap,MSNbibl,seceqn,dvips]{arximspdf}

\doi{10.1214/14-AAP1015} 
\volume{25}
\issue{2}
\pubyear{2015}
\firstpage{974}
\lastpage{985}

\makeatletter
\newcommand{\ind}{\mathbf{1}}

\newtheorem{teo}{Theorem}[section]
\newtheorem{lem}{Lemma}[section]
\newproclaim{defn}{Definition}[section]
\newproclaim{example}{Example}[section]
\newproclaim{alg}{Algorithm}[section]
\makeatother

\begin{document}
\begin{frontmatter}

\title{Approximation algorithms for the normalizing constant of Gibbs
distributions}
\runtitle{Approximation of partition functions}

\begin{aug}
\author[A]{\fnms{Mark}~\snm{Huber}\corref{}\ead[label=e1]{mhuber@cmc.edu}\ead[label=u1,url]{http://www.cmc.edu/pages/faculty/MHuber/}}
\runauthor{M. Huber}
\affiliation{Claremont McKenna College}
\address[A]{Department of Mathematical Sciences\\
Claremont McKenna College\\
850 Columbia Avenue\\
Claremont, California 91711\\
USA\\
\printead{e1}\\
\printead{u1}} 
\end{aug}

\received{\smonth{6} \syear{2012}}
\revised{\smonth{1} \syear{2014}}

%
\begin{abstract}
Consider a family of distributions $\{\pi_\beta\}$
where $X \sim\pi_\beta$ means that
$\mathbb{P} (X = x) = \exp(-\beta H(x)) / Z(\beta)$. Here $Z(\beta)$
is the proper normalizing constant, equal to
$\sum_x \exp(-\beta H(x))$. Then $\{\pi_\beta\}$ is known as
a Gibbs distribution, and $Z(\beta)$ is
the partition function. This work presents a new method for
approximating the partition function to a specified level of relative
accuracy using only a number of samples,
that is,
$O(\ln(Z(\beta)) \ln(\ln(Z(\beta))))$ when $Z(0) \geq1$.
This is a sharp improvement
over previous, similar approaches that used a much more complicated algorithm,
requiring $O(\ln(Z(\beta)) \ln(\ln(Z(\beta)))^5)$ samples.
\end{abstract}

%

\end{frontmatter}

\section{Introduction}

The central idea of Monte Carlo methods is that the ability to sample
from certain distributions gives a means for estimating the value of
an integral or sum. This paper presents
a new method for using samples to approximate
a broad class of sums coming from Gibbs distributions that is
faster than previously-known methods.

%
\begin{defn} $\{\pi_\beta\}_{\beta\in{\mathbb R}}$
is a \textit{Gibbs distribution with parameter
$\beta$} over finite state space
$\Omega$ if there exists a \textit{Hamiltonian function}
$H(x)\dvtx \Omega\rightarrow\mathbb{R}$
such that
for $X \sim\pi_\beta$,
\[
\mathbb{P}(X = x) = \exp\bigl(-\beta H(x)\bigr) / Z(\beta),
\]
where $Z(\beta) = \sum_{x \in\Omega} \exp(-\beta H(x))$ is called
the \textit{partition function} of the distribution.
\end{defn}

The partition function can be difficult to compute, even when
dealing with simple
problems.

%
\begin{example}[(The Ising model)]\label{EXPising}
Given a graph $G = (V,E)$, let
$\Omega= \{-1,1\}^V$, and $H(x) = -\sum_{\{i,j\} \in E} \ind(x(i) = x(j))$,
where $\ind(\cdot)$ is the indicator function that is 1 if the argument
is true and 0 if it is false.
Then the Gibbs distribution with this Hamiltonian is called
the \textit{Ising model}. Finding $Z(\beta)$ for arbitrary graphs is
a \#P-complete problem~\cite{jerrums1993}.
\end{example}

A vast literature has arisen devoted to finding ways to generate
random variables from Gibbs distributions; see, for
instance, \cite{huber2010a,huber2004a,metropolisrrtt1953,swendsenw1986} or
\cite{brooksgjm2011} for an overview.
For
the Ising model, Jerrum and Sinclair~\cite{jerrums1993} give an
algorithm for approximately sampling from $\pi_\beta$ in polynomial time
for $\beta> 0$. Propp and Wilson~\cite{proppw1996}
give an algorithm for the Ising model
that seems to run efficiently when $\beta> 0$ is at or below
a cutoff known as the critical value.

Once an effective method for
obtaining approximate or perfect samples from the target Gibbs distribution
exists, the question becomes:
what is the best way of using those samples to approximate
$Z(\beta)$?

%
\begin{defn}
Say that $\mathcal A$ is an
$(\varepsilon,3/4)$-\textit{randomized approximation algorithm}
for $Z(\beta)$ if
it outputs value $\hat Z(\beta)$ such that
\[
\mathbb{P} \biggl(\frac{1}{1 + \varepsilon} \leq\frac{\hat Z(\beta
)}{Z(\beta)} \leq1 + \varepsilon
\biggr) \geq3/4.
\]
\end{defn}

Here $\varepsilon\geq0$
controls the relative error between the approximation and
the true answer.
The $3/4$ on the right-hand side
can be made arbitrarily close to 1 by repeating the
algorithm and taking the median of the resulting output.

\subsection{Previous work} The first step in building such an
approximation algorithm is importance sampling. For most Gibbs distributions,
calculating $Z(0)$ is straightforward, and it is easy to
generate samples from $\pi_0$. For the Ising model, $\pi_0$ is just
the uniform distribution over $\{-1,1\}^V$, and
$Z(0) = 2^{\# V}$. With a draw $X \sim\pi_0$
in hand, let
%
%
\begin{equation}
\label{EQNw} W = \exp\bigl(-\beta H(X)\bigr).
\end{equation}
Then
\[
\mathbb{E}[W] = \frac{\sum_{x \in\Omega} \exp(-\beta H(x)) \exp(0)}{Z(0)} = \frac{Z(\beta)}{Z(0)},
\]
making $W \cdot Z(0)$ an unbiased estimator of $Z(\beta)$.

The relative performance of this Monte Carlo estimate is controlled by the
relative variance, the square of the coefficient of variation.
For a random variable $X$ with finite second moment,
$\mathbb{V}_{\mathrm{rel}}(X) = [\mathbb{E}(X^2)/\mathbb{E}(X)^2] - 1$.
Hence for the random variable $W$ as in (\ref{EQNw}),
%
%
\begin{equation}
\label{EQNrelvarw}\qquad \mathbb{V}_{\mathrm{rel}}(W) = -1 + \frac{\sum_{x \in\Omega} \exp(-\beta H(x))^2}{Z(0)}\cdot
\frac{Z(0)^2}{Z(\beta)^2} = -1 + \frac{Z(2 \beta) Z(0)}{Z(\beta)^2}.
\end{equation}
There are two main issues with this relative variance:
\begin{longlist}[(2)]
\item[(1)]{For problems like the Ising model, this last
ratio can be exponentially large in the input, making the method untenable.}
\item[(2)]{The relative variance involves the value of $Z(2\beta)$,
outside the
interval of interest $[0,\beta]$. Typically, larger values of $\beta$
make sampling from $\pi_\beta$ more difficult. This presents a serious
impediment to the method.}
\end{longlist}
The first problem can be dealt with by using the \textit{multistage sampling}
method of
Valleau and Card~\cite{valleauc1972}.
In this approach, a sequence of
$\beta$ values $0 = \beta_0 < \beta_1 < \beta_2 < \cdots< \beta
_\ell= \beta$
are introduced, called a \textit{cooling schedule}. Then
\[
\frac{Z(\beta)}{Z(0)} = \frac{Z(\beta_1)}{Z(\beta_0)} \cdot \frac{Z(\beta_2)}{Z(\beta_1)} \cdots
\frac{Z(\beta_\ell)}{Z(\beta
_{\ell- 1})}.
\]
Each of the individual factors in the product on the right can then be
estimated separately and then multiplied to give a final estimate.
Fishman calls an estimate of this form a \textit{product estimator}
\cite{fishman1994}, page 437.

It is straightforward to calculate the mean and
relative variance of a product
estimator in terms of the mean and relative variance of the individual
factors. The following result is a simplified form of a result that appears
on page 136 of~\cite{dyerf1991}.

%
\begin{lem}[(\cite{dyerf1991})]
\label{LEMproductrelvar}
For $P = \prod P_i$ where the $P_i$ are independent,
\[
\mathbb{E}[P] = \prod\mathbb{E}[P_i], \qquad
\mathbb{V}_{\mathrm{rel}}(P) = -1 + \prod\bigl(1 + \mathbb{V}_{\mathrm
{rel}}(P_i)
\bigr).
\]
\end{lem}
%

Let $q = \ln(Z(\beta)/Z(0))$, and suppose
$H(x) \in\{0,\ldots,n\}$.
Next, Bez\'akov\'a et al.~\cite{bezakovasvv2008}
introduce a fixed cooling schedule with two pieces, the first where
the parameter value grows linearly and the second where it grows exponentially,
\[
0,\frac{1}{n},\frac{2}{n},\ldots,\frac{k}{n},
\frac{k\gamma}{n}, \frac{k\gamma^2}{n},\ldots,\frac{k\gamma^t}{n},
\]
where\vspace*{1pt} $k = \lceil q \rceil$ and $\gamma= 1 + 1/q$. With this fixed
cooling schedule, they give an $(\varepsilon,3/4)$-approximation algorithm
that uses $O(q^2 (\ln n)^2)$ samples in the worse case.

By using an adaptive cooling schedule, it is possible to do better.
In~\cite{stefankovicvv2009}, \u{S}tefankovi\u{c}, Vempala and Vigoda
introduce an adaptive cooling schedule.
Their algorithm
is highly complex, and they are interested primarily in the
asymptotic order of the running time
rather than a practical implementation. Their
$(\varepsilon,3/4)$-approximation algorithm uses, at most,
%
%
\begin{equation}
\label{EQNsvv} 10^8 q\bigl(\ln(n) + \ln(q)\bigr)^5
\varepsilon^{-2}\bigr)
\end{equation}
samples on average from the target distribution.

In~\cite{huber2010b}, the Huber and Schott introduce a general technique
for finding normalizing constants of sums and integrals
called \textsf{TPA}. When applied to the specific problem area of
Gibbs distributions, the running time
for an $(\varepsilon,3/4)$-approximation algorithm becomes $O(q^2)$.
While this algorithm is much simpler to implement than the
method of {S}tefankovi\u{c}, Vempala and Vigoda~\cite{stefankovicvv2009}, it has a worse running time, asympototically.

\subsection{Main result}
The multistage idea solves the issue of $Z(2\beta)Z(0)/\break Z(\beta)^2$ being
too large, but fails to solve the issue of the variance depending on
$Z(2\beta)$. Dealing with this leads to several of the $\ln$ factors
in~\cite{stefankovicvv2009}. In this work a new method is introduced, the
\textit{paired product estimator}, which has a variance only involving
quantities within $[0,\beta]$. The result is an algorithm where the
overal variance can be analyzed precisely. This allows for the construction
of an approximation
algorithm much simpler than that found in \cite{stefankovicvv2009}, and
which requires far fewer samples.

%
\begin{teo}
\label{THMintro}
Suppose $n \geq4$ and $\varepsilon\leq1/10$.
When $H(x) \in\{0,1,\ldots,n\}$ or $\{0,-1,-2,\ldots,-n\}$,
the new method is an $(\varepsilon,3/4)$-approximation algorithm that uses
only
%
%
\begin{equation}
\label{EQNsamples} \qquad (q + 1)\bigl[5 + \bigl(2 + \ln(2n)\bigr) \bigl(14.9 \ln\bigl(100
\bigl(2 + \ln(2n)\bigr) (q+1)\bigr) + 48.2 \varepsilon^{-2}\bigr)\bigr]
\end{equation}
and draws from the Gibbs distribution on average.
\end{teo}

It is, of course, possible to derive an upper bound on the number of
samples used
when $n < 4$ or $\varepsilon> 1/10$; however, adding these assumptions
makes the presentation cleaner.

The requirement that $H(x) \in\{0,\ldots,n\}$ or $\{-n,\ldots,0\}$ is
so that $H(x)$ does not change sign, which is a necessary condition
for the algorithm. Suppose that
$H(x) \in\{a,a+1,\ldots,a + n\}$ where $a$ is known. Then
using $H'(x) = H(x) - a$ gives the same Gibbs distribution as with $H$, so
drawing samples from $H'$ is no more difficult than drawing
from $H$ and $H'(x) \in\{0,\ldots,n\}$. However,
the partition function is different. If $Z(\beta)$ was the original
partition function, and $Z_{H'}(\beta)$ the new, then
$Z_{H'}(\beta) = \exp(\beta a) Z(\beta)$. Hence
$q'$ for $H'$ satisfies $q' = q + a\beta$. Theorem~\ref{THMintro}
can then be applied.

Section~\ref{SECalgorithm} describes the overall structure of the algorithm
and shows how to obtain a good cooling schedule.
Section~\ref{SECanalysis}
then analyzes the relative variance of the pieces of the algorithm
in order to prove Theorem~\ref{THMintro}.

\section{The algorithm}\label{SECalgorithm}

Let $q = \ln(Z(0)/Z(\beta))$. Then to obtain an approximation within
a factor
of $1 + \varepsilon$ of $Z(0)/Z(\beta)$, it is necessary to obtain an
approximation of $q$ within an additive factor of $\ln(1 + \varepsilon)$.
The main algorithm consists of the following pieces:
\begin{longlist}[(3)]
\item[(1)]{obtain an initial estimate of $q$;}
\item[(2)]{obtain a well-balanced cooling schedule;}
\item[(3)]{use the well-balanced schedule with the paired product estimator.}
\end{longlist}

Let $z(\beta) = \ln(Z(\beta))$.
Then \textit{well-balanced} means that there exists $\eta\geq0$ such that
$|z(\beta_{i+1}) - z(\beta_{i})| \leq\eta$ for all $i$.

The first two pieces will be accomplished using \textsf{TPA}, introduced
in~\cite{huber2010b}. To use \textsf{TPA} for Gibbs distributions on
parameter values $[0,\beta]$, it is necessary that $H(x)$ be either
always nonnegative or always nonpositive.

In the Ising model example shown earlier, $H(x) \leq0$, and so
$Z(\beta)$ is an increasing function of $\beta$.
In this case,
\textsf{TPA} is an algorithm that generates a random set of parameter
values in
the interval from 0 to $\beta$ by taking samples from $\pi_b$ for
various values of $b \in[0,\beta]$.
Then
the output of \textsf{TPA} is a
Poisson point process (PPP) of rate 1 in $[z(0),z(\beta)]$; see Section~2
of~\cite{huber2010b}.

%
\begin{alg}
\label{ALGtpagibbs}
\textsf{TPA} for Gibbs distributions with $H(x) \leq0$
takes as input a value $\beta> 0$ together with an oracle for generating
random samples from $\pi_b$ for $b \in[0,\beta]$,
and returns a set of values
$0 < b_1 < b_2 < \cdots< b_\ell< b$ such that
$\{z(b_1),\ldots,z(b_\ell)\}$ forms a Poisson point process of rate 1 on
the interval $[z(0),z(\beta)]$. It operates as follows:
\begin{longlist}[(3)]
\item[(1)]{start with $b$ equal to $\beta$ and $B$ equal to the empty set;}
\item[(2)]{draw a random sample $X$ from $\pi_b$, and draw $U$ uniformly
from $[0,1]$;}
\item[(3)]{let $b = b - \ln(U)/H(X)$, unless $H(X) = 0$, in which
case set
$b = -\infty$;}
\item[(4)]{if $b > 0$, then add $b$ to the set $B$, and go back to
step 2.}
\end{longlist}
\end{alg}

The number of samples drawn by \textsf{TPA} will equal 1 plus a Poisson
random variable with mean $q$ \cite{huber2010b}, pages 3--4.
The output of Algorithm~\ref{ALGtpagibbs}
can be used in several different ways.
When \textsf{TPA} is run
$k$ times and the output sets combined, and
the result is a Poisson point process on
$[z(0),z(\beta)]$ of rate $k$.

It is even possible to obtain rates that are fractional. To obtain rate
$k$ where $k$ is not an integer, first run \textsf{TPA} $\lceil k \rceil
$ times.
Then for each point of the process, keep it independently with
probability $k / \lceil k \rceil$. Otherwise discard it entirely.
This procedure, known as \textit{thinning}, enables creation of a PPP
of any positive rate, which will simplify the analysis later; see \cite{resnick1992}, page 320, for more on thinning.

After a PPP of rate $k$ has been generated, the number of points in the
process has a Poisson distribution with mean
$k(z(\beta) - z(0))$. This gives a way of initially getting an estimate
of $z(\beta) - z(0)$ that (by choosing $k$ high enough) has a 99\% chance
of being within a factor of 2 of the correct value.

Once that is accomplished, \textsf{TPA} is run, this time with an even larger
value of $k$ based on the estimate from the first step. Because the $z(b)$
values form a Poisson point process, the difference between successive $z(b)$
values will be an exponential random variable, so if $b'$ is the $d$th point
following $b$, then $z(b') - z(b)$ will have a gamma (Erlang) distribution
with shape parameter $d$ and rate parameter $k$. By making $k$ and $d$ large
enough, this will be tightly concentrated around its mean value of $d/k$
for all such differences.
The result is a set of parameter values $\{\beta_i\}$
that are well balanced.

Call $[\beta_i,\beta_{i+1}]$ interval $i$.
Now each $z(\beta_{i+1}) - z(\beta_i)$ will be estimated
independently using
the paired product estimator. This works as follows. For each interval~$i$, let $m_i = (\beta_i + \beta_{i+1})/2$ be the midpoint of the
interval, and
$h_i = m_i - \beta_i = \beta_{i+1} - m_i$ be
the half length of an interval.
Draw $X \sim\pi_{\beta{i}}$ and $Y \sim\pi_{\beta_{i+1}}$. Then set
\[
W_i = \exp\bigl(-h_i H(X)\bigr), \qquad V_i
= \exp\bigl(h_i H(Y)\bigr).
\]

Then
\[
\mathbb{E}[W_i] = \frac{\sum\exp(-\beta_i H(x))\exp(-h_i
H(x))}{Z(\beta_i)} = \frac{\sum\exp(-m_i H(x))}{Z(\beta_i)} =
\frac{Z(m_i)}{Z(\beta_i)}.
\]
Similarly, $\mathbb{E}[V_i] = Z(m_i)/Z(\beta_{i+1})$.
Therefore, $W_i$ can
be used to estimate the drop $z(m_i) - z(\beta_i)$, and
$V_i$ can estimate the drop $z(\beta_{i+1}) - z(m_i)$.

Now we have the relative variance calculation.
\begin{eqnarray*}
\mathbb{V}_{\mathrm{rel}}(W_i) &=& \frac{\mathbb{E}[W_i^2]}{\mathbb
{E}[W_i]^2} - 1 = -1 +
\frac{\sum\exp(-\beta_i H(x)) \exp(-\delta_i H(x))^2} {
Z(\beta_i)} \cdot \frac{Z(\beta_i)^2}{Z(m_i)^2}
\\
&=& -1 + \frac{Z(\beta_{i+1}) Z(\beta_i)}{Z(m_i)^2}\qquad\mbox{since } \beta_i + 2
\delta_i = \beta_{i+1}.
\end{eqnarray*}
A similar calculation shows that $\mathbb{V}_{\mathrm{rel}}(V_i) =
\mathbb{V}_{\mathrm{rel}}(W_i)$, and
now the variance of our estimators for interval $i$ only involves
$Z(b)$ values for $b$ that fall in interval $i$.

Let $W$ be the product of the $W_i$ over all intervals $i$, and
$V$ be the product of the~$V_i$. Then the final estimate of
$Z(\beta)/Z(0)$ is $W / V$. This is not quite an unbiased estimator, but
it is true that $\mathbb{E}[W]/\mathbb{E}[V] = Z(\beta)/Z(0)$. If
both $W$ and $V$
are tightly concentrated around their means, then $W/V$ will be close to
$Z(\beta)/Z(0)$. To get that tight concentration, in the next section
it is shown that the relative variance of~$W$ (and $V$) is small as long
at the $\beta$ values form a well-balanced schedule.

With that small relative variance,
it is possible to repeatedly draw independent, indentical copies of $W$
to get a sample average $\bar W$ which is tightly concentrated about
its mean. (The same is true for $V$ as well.) The
following algorithm incorporates these ideas.

\begin{alg}[(Paired product approximation algorithm)]
\label{ALGppaa}
The input is a value $\beta> 0$ together with an oracle for generating
samples from $\pi_b$ for $b \in[0,\beta]$. The output is an approximation
for $Z(\beta)/Z(0)$.
\begin{longlist}
\item[(1)]{Run \textsf{TPA} 5 times to get an estimate of
$q = \ln(Z(\beta)/Z(0))$ that is at least $q/2$ with probability 99\%.}
\item[(2)]{Run \textsf{TPA} $k$ times to obtain a set of parameter values.
Sort these values and then keep every $d$th successive value.
Add parameter
values $0$ and $\beta$, and label the result
$0 = \beta_0 < \beta_1 < \cdots< \beta_\ell= \beta$.}
\item[(3)]{Repeat the following
$\lceil2e \sqrt{10} ((1 + \varepsilon)^{1/2} - 1)^{-2} \rceil$
times:
for each $i$, draw $X_i \sim\pi_{\beta_i}$, let $W_i = \exp(-\delta
_i H(X_i))$
and\vspace*{2pt} $V_i = \exp(\delta_i H(X_{i+1}))$, $W = \prod W_i$ and $V = \prod V_i$.
Take the sample average of the $W$ values to get $\bar W$, and the
sample average of the $V$ values to get $\bar V$.}
\item[(4)]{The estimate of $Z(\beta) / Z(0)$ is $\bar W/\bar V$.}
\end{longlist}
\end{alg}

Note that $((1 + \varepsilon)^{1/2} - 1)^{-2} \approx4\varepsilon^{-2}$.
It is
necessary to use this more complex expression because the final estimator
is the ratio of $W$ and $V$; see the proof of Theorem~\ref{THMrunningtime}.
Algorithm~\ref{ALGppaa} can be run for any values of $d$ and $k$. The next
section shows how to choose them properly to make Algorithm~\ref{ALGppaa}
an \mbox{$(\varepsilon,3/4)$-}approximation algorithm.

\section{Analysis}\label{SECanalysis}

In this section the following theorem is shown.

\begin{teo}
\label{THMmain}
In Algorithm~\ref{ALGppaa}, let $\hat q_1$ be the size of the
Poisson point process created with 5 runs of \textsf{TPA} in step 1.
Let
\[
d = \bigl\lceil22 \ln\bigl(100\bigl(2 + \ln(2n)\bigr) (\hat q_1 +
1/2)\bigr) \bigr\rceil\quad \mbox{and}\quad k = (2/3)d \bigl[2 + \ln(2n)\bigr].
\]
Then the algorithm output is within $1 + \varepsilon$ of $Z(\beta)/Z(0)$
with probability at least~3/4.
\end{teo}

Let $q = \ln(Z(\beta)/Z(0))$.
The proof breaks into three parts. The first shows that by running
\textsf{TPA} 5 times, the probability that $\hat q_1 + 1/2 < (1/2)q$ is
at most
1\%. The second part shows that with the choice of $k$, the
probability that the schedule is not well balanced is at most 4\%. Finally,
the third part shows that the third step of the algorithm produces
$\bar W$ and $\bar V$ that are both within $1 + \tilde\varepsilon/2$ of
their respective means with probability at most 20\%. The union bound
on the
probability of failure is then $1\% + 4\% + 20\% = 25\%$, as desired.

\subsection{The initial estimate \texorpdfstring{$\hat{q}_1$}{$q_1$}}
Recall that Algorithm~\ref{ALGtpagibbs} has output that is a Poisson
point process with rate 1. Let $k_1$ denote the number of times that
\textsf{TPA} is run and the output combined. Then the new PPP has a
rate of
$k_1$. Therefore the
number of points in the PPP
is Poisson distributed with
mean $k_1(z(\beta) - z(0))$. The following lemma concerning Poisson
random variables then shows that $\hat q_1 + 1/2$ is at least $1/2$ of its
mean with probability at least 99\%.

%
\begin{lem}
\label{LEMpoisson}
Let $X$ have Poisson distribution with mean $\mu$. Then
$\mathbb{P}(X < \mu/2) \leq2(\pi\mu)^{-1/2}(2/e)^{\mu/2}$.
\end{lem}

\begin{pf}
Suppose $\mu/2 = \lceil\mu/2 \rceil$. Then
\[
\mathbb{P}(X < \mu/2) = \exp(-\mu)\sum_{i \leq\mu/2}
\frac{\mu
^i}{i!} \leq \exp(-\mu) 2 \frac{\mu^{\mu/2}}{(\mu/2)!}.
\]
The last inequality comes from the fact that each term in the sum is
at least twice the previous term. The Stirling bound
$i! > \sqrt{2\pi i}(i/e)^i$ gives
$\mathbb{P}(X \leq\mu/2) \leq2(\pi\mu)^{-1/2}(2/e)^{\mu/2}$. Now
suppose $\mu/2 \neq\lceil\mu/2 \rceil$. Let
$\mu' = 2\lceil\mu/2 \rceil$.
\[
\mathbb{P}(X < \mu/2) \leq\mathbb{P}\bigl(X \leq\mu' / 2\bigr)
\leq2\bigl(\pi\mu'\bigr)^{-1/2}(2/e)^{\mu'/2} \leq2(\pi
\mu)^{-1/2}(2/e)^{\mu}.
\]\upqed
\end{pf}

Suppose step 1 runs $k_1$ repetitions of \textsf{TPA}.
Then $\hat q_1$ has a Poisson distribution with mean
$k_1 q$. If $q \leq1$, then it is always true
that $\hat q_1 + 1/2 \geq(1/2)q$.
If $q > 1$, then setting $k_1 = 5$ and using Lemma~\ref{LEMpoisson}
makes the probability of
failure below~$1\%$.

\subsection{The well-balanced schedule}\label{SBSwellbalanced}
Now consider the second step in Algorithm~\ref{ALGppaa}. First, run
\textsf{TPA} $k$ times to get a set $B$ that is a PPP of rate $k$ on
the interval $[z(0),z(\beta)]$.
Since $B$ is a PPP of rate $k$,
if $b < b'$ are values in $B$ such that there are exactly $d - 1$
values in
$(b,b')$, then $z(b') - z(b)$ has a gamma distribution with parameters
$d$ and $k$. This is equivalent to saying
$z(b') - z(b)$ has the distribution of the sum of
$d$ independent exponential random variables each with rate $k$.
Hence the
moment generating function of $z(b') - z(b)$ is $[k/(k - t)]^d$. Let
$t$ and $\eta$ be nonnegative real numbers, then
\begin{eqnarray*}
&& \mathbb{P}\bigl(z\bigl(b'\bigr) - z(b) \geq\eta\bigr)
\\
&&\qquad =
\mathbb{P}\bigl(\exp\bigl(t\bigl(z\bigl(b'\bigr) - z(b)\bigr)\bigr) \geq\exp(
\eta t)\bigr)
\\
&&\qquad = \bigl[k/(k-t)\bigr]^d \exp(-\eta t)\qquad\mbox{by
Markov's inequality}
\\
&&\qquad = (\eta k /d)^d \exp(-\eta k + d)\qquad\mbox{by setting } t = k - d
/ \eta.
\end{eqnarray*}
%
On the other hand, for $t > 0$, multiplying by $-t$ and exponentiating gives
\begin{eqnarray*}
&& \mathbb{P}\bigl(z\bigl(b'\bigr) - z(b) \leq\eta/2\bigr)
\\
&&\qquad =
\mathbb{P}\bigl(\exp\bigl(-t\bigl(z\bigl(b'\bigr) - z(b)\bigr)\bigr) \geq\exp(-
\eta t/2)\bigr)
\\
&&\qquad = \bigl[k/(k+t)\bigr]^d \exp(\eta t/2)\qquad\mbox{by
Markov's inequality}
\\
&&\qquad = \bigl(\eta k/(2d)\bigr)^d \exp(- \eta k / 2 + d)\qquad\mbox{by
setting } t = 2d/\eta- k.
\end{eqnarray*}
%
%
So if $d = (3/4)\eta k$, then from the union bound
\[
\mathbb{P}\bigl(\eta/2 \leq z\bigl(b'\bigr) - z(b) \leq\eta\bigr)
\geq1 - \bigl[\exp(-1/3)\cdot4/3\bigr]^d - \bigl[\exp(1/3)\cdot2/3
\bigr]^d.
\]
For the PPP,
the chance that $z(b) - z(b') \in[\eta/2,\eta]$
for the first $2\eta^{-1}(z(\beta) - z(0))$
intervals to the left of $\beta$ is (again by the union bound)
at least $1 - 2\eta^{-1}(z(\beta) - z(0)) 2 [\exp(-1/3)\cdot4/3]^d$.
Making
\[
d \geq\frac{\ln(0.04(4\eta^{-1}(z(\beta)-z(0)))^{-1})}{-(1/3) +
\ln(4/3)} = \frac{\ln(100\eta^{-1}(z(\beta) - z(0)))}{1/3 - \ln(4/3)}
\]
would make this probability
at least 96\%. However, $q = z(\beta) - z(0)$ is unknown. What is known
(from step 1 of Algorithm~\ref{ALGppaa} is $2(\hat q_1 + 1/2)$ has a
96\% chance of being at least $q$. Since
$(1/3 - \ln(4/3))^{-1} = 21.905,\ldots,$ setting
\[
d = \bigl\lceil22 \ln\bigl(200\eta^{-1}(\hat q + 1/2)\bigr) \bigr
\rceil
\]
and $k = (4/3)d/\eta$
makes the chance that step 2 fails to find a schedule where
$z(b) - z(b') > 1$ for any interval at most 4\%.

\subsection{Choosing \texorpdfstring{$\eta$}{$eta$}}
The next question to consider is the size of $\eta$. The value
of $\eta$ will be used to control the overall relative variance of
the product estimators $W$ and $V$.
%
For the $i$th interval $[\beta_i,\beta_{i+1}]$, let
$m_i \stackrel{\mathrm{def}}{=}(\beta_i + \beta_{i+1}) /2$ be the
midpoint of the interval.
Let $\delta_i$ be the difference between the $y$-coordinate of the
midpoint of the interval
secant line and the function value at the midpoint of the interval.
That is,
\[
\delta_i \stackrel{\mathrm{def}} {=}\frac{z(\beta_{i+1}) + z(\beta
_i)}{2} -
z(m_i).
\]
From (\ref{EQNrelvarw}), $\mathbb{V}_{\mathrm{rel}}(W_i) = \exp
(2\delta_i) - 1$. Since
the relative variance is always nonnegative, this implies that
$\delta_i \geq0$ and so the function $z$ is convex.

From Lemma~\ref{LEMproductrelvar},
%
%
\begin{equation}
\label{EQNbound} \mathbb{V}_{\mathrm{rel}}(W) = -1 + \prod\bigl(1 +
\exp(2\delta_i) - 1\bigr) = -1 + \exp \Bigl(\sum2\delta_i \Bigr).
\end{equation}

So controlling the overall relative variance is a matter of
bounding $\delta_i$ for each interval $i$.
The key idea in the bound comes from~\cite{stefankovicvv2009}, although
they use it in a very different fashion. The idea
is that when $\delta_i$ is large, the derivative of $z$
sharply increases.

%
\begin{lem}
For the $i$th interval $[\beta_i,\beta_{i+1}]$ with
$z(\beta_{i+1}) - z(\beta_i) = \eta_i$,
\[
\frac{z'(\beta_{i+1})}{z'(\beta_i)} \geq\exp(4\delta_i / \eta_i).
\]
\end{lem}

\begin{pf}
Let $m_i = (\beta_i + \beta_{i+1})/2$ be the midpoint of interval $i$,
and $\eta_i = z(\beta_{i+1}) - z(\beta_{i})$ be the change in the
$z$ function over the interval.
Since $z$ is convex, the slope at $\beta_i$ is at most
$[z(m_i) - z(\beta_i)]/[m_i - \beta_i]$. On the other hand, the slope at
$\beta_{i+1}$ is at least $[z(\beta_{i+1}) - z(m_i)]/[\beta_{i+1} - m_i]$.
Since $m_i$ is\vadjust{\goodbreak} the midpoint of the interval,
$m_i - \beta_i = \beta_{i+1} - m_i$ and
\[
\frac{z'(\beta_{i+1})}{z'(\beta_{i})} \geq\frac{z(\beta_{i+1}) - z(m_i)} {
z(m_i) - z(\beta_{i})} = \frac{\eta_i/2 + \delta_i}{\eta_i/2 - \delta_i} =
\frac{1 + 2\delta_i/\eta_i}{1 - 2\delta_i/\eta_i} \geq\exp(4\delta_i/\eta_i).
\]\upqed
\end{pf}

%
\begin{lem}
\label{LEMbound}
For a cooling schedule over $[0,\beta]$
with $z(\beta_{i+1}) - z(\beta_{i}) \leq\eta$ for all $i$,
\[
\mathbb{V}_{\mathrm{rel}}(W) = \mathbb{V}_{\mathrm{rel}}(V) \leq \cases{2, &
\quad$z'(\beta) < 1/2$,
\vspace*{3pt}\cr
\bigl(2z'(\beta)
\bigr)^{\eta/2}, &\quad$z'(0) \geq1/2$,
\vspace*{3pt}\cr
2 e^{\eta}
\bigl[2 z'(\beta)\bigr]^{\eta/2}, &\quad$z'(0) <
1/2 \leq z'(\beta)$.}
\]
For $n \geq4$ and $\eta= 2 / [2 + \ln(2n)]$, regardless of
$z'(0)$ and $z'(\beta)$,
\[
\mathbb{V}_{\mathrm{rel}}(W) = \mathbb{V}_{\mathrm{rel}}(V) \leq2e.
\]
\end{lem}

\begin{pf} Recall that $\mathbb{V}_{\mathrm{rel}}(W) \leq\exp(2\sum_i \delta_i)$ so
the goal is to bound $\sum_i \delta_i$.

Consider a cooling schedule
$0 = \beta_0 < \beta_1 < \cdots< \beta_\ell= \beta$. It is well known
that $z'(\beta)$ is just $\mathbb{E}[-H(X)]$ where $X \sim\pi_\beta$
\[
z'(\beta) = \frac{d}{d\beta} \ln\bigl(Z(\beta)\bigr) =
\frac{Z'(\beta
)}{Z(\beta)} = \frac{\sum_x -H(x)\exp(-\beta H(x))}{Z(\beta)} = \mathbb{E}\bigl[-H(X)\bigr].
\]
\begin{longlist}[\textit{Case} III]
\item[\textit{Case} I:] $z'(\beta) < 1/2$. Then $H(x) \leq-1 \Longrightarrow-H(x) \geq
1$ so
\begin{eqnarray*}
&& \frac{\sum_{x\dvtx H(x) \leq-1} -H(x) \exp(-\beta H(x))}{Z(\beta)} \leq \frac{1}{2}
\\
&&\qquad \Rightarrow\quad \frac{\sum_{x\dvtx H(x) \leq-1} \exp(-\beta H(x))}{Z(\beta)} \leq
\frac
{1}{2}
\\
&&\qquad \Rightarrow\quad \frac{\sum_{x\dvtx H(x) = 0} \exp(-\beta H(x))}{Z(\beta)} \geq \frac{1}{2}
\\
&&\qquad \Rightarrow\quad \frac{Z(0)}{Z(\beta)} \geq\frac{1}{2}.
\end{eqnarray*}
Hence $z(\beta) - z(0) \leq\ln(2)$ which means
$\sum_i 2\delta_i \leq\ln(2)$ and $\exp(\sum_i 2\delta_i) \leq2$.

\item[\textit{Case} II:] $z'(0) \geq1/2$. Then $2z'(\beta) \geq z'(\beta)/z'(0)$, and
from the last lemma
\[
\frac{z'(\beta)}{z'(0)} = \frac{z'(\beta_1)}{z'(\beta_0)}\cdots \frac{z'(\beta_{\ell})}{z'(\beta_{\ell-1})} \geq\prod
_i \exp(4 \delta_i /
\eta_i).
\]
Raising to the $\eta/ 2$ power then finishes this case.

\item[\textit{Case} III:] $z'(0) < 1/2 \leq z'(\beta)$. Since $z'$ is continuous,
let $a \in[0,\beta]$
be the parameter value where
$\mathbb{E}[-H(X)] = 1/2$ for $X \sim\pi_a$, and suppose $a$ is in
the $j$th
interval $[\beta_j,\beta_{j+1}]$. As in case~I, $Z(\beta_j)/Z(\beta
_0) \leq2$.
As in case~II,\vadjust{\goodbreak} $\prod_{i > j} \exp(4\delta_i) \leq[2z'(\beta
)]^{\eta/2}$.
Since $2\delta_j \leq\eta$, this means that the combined relative
variance is
at most $2 e^{\eta} [2 z'(\beta)]^{\eta/2}$.

Since $z'(\beta) = \mathbb{E}[-H(X)]$ for $X \sim\pi_\beta$, and
$X \leq n$,
$z'(\beta) \leq n$. Hence if $\eta/2 \leq1/[2 + \ln(2n)]$, then
$e^\eta[2z'(\beta)]^{\eta/2} \leq e$.\quad\qed
\end{longlist}\noqed
\end{pf}

\begin{pf*}{Proof of Theorem~\ref{THMmain}}
Using the value of $d$ from Section~\ref{SBSwellbalanced} and
Lemma~\ref{LEMbound} gives that the relative variance for an instance
of $W$ (or $V$) is at most~$2e$. All that remains is to analyze the third
step of Algorithm~\ref{ALGppaa}. It is easy to verify that
if $\bar W$ is the sample average of $r$ independent, identically distributed
(i.i.d.)
instances of $W$, then $\mathbb{V}_{\mathrm{rel}}(\bar W) = \mathbb
{V}_{\mathrm{rel}}(W)/r$.
Let $\tilde\varepsilon= (1 + \varepsilon)^{1/2} - 1$. For
$\lceil2e\sqrt{10} \tilde\varepsilon^{-2} \rceil$ i.i.d. draws of $W$,
$\mathbb{V}_{\mathrm{rel}}(\bar W) \leq\tilde\varepsilon^{-2}/10$.

Chebyshev's inequality says that for a random variable $X$ with
finite relative variance,
$\mathbb{P}((1 - \varepsilon)\mathbb{E}[X] \leq X \leq(1 + \varepsilon) X)
\geq1 - \mathbb{V}_{\mathrm{rel}}(X) \varepsilon^2$. Hence
\[
\mathbb{P}\bigl((1 + \tilde\varepsilon)^{-1} \mathbb{E}[W] \leq\bar W
\leq (1 + \tilde\varepsilon) \mathbb{E}[W]\bigr) \geq1 - 1/10.
\]
Similarly,
$\mathbb{P}((1 + \tilde\varepsilon)^{-1} \mathbb{E}[V] \leq\bar V
\leq
(1 + \tilde\varepsilon) \mathbb{E}[V]) \geq1 - 1/10$.

Therefore, the chance that step 1 successfully
gives a basic estimate of $\ln(Z(\beta)/Z(0))$, step 2 creates a well-balanced
schedule and step 3 gives $\bar W$ and $\bar V$ both within a factor of
$(1 + \tilde\varepsilon)$ of their respective means is at
least $1 - 1/100 - 4/100 - 1/10 - 1/10 = 75\%$ by the union bound.

If both $\bar W$ and $\bar V$ are within $1 + \tilde\varepsilon$ of their
means, then $\bar W/\bar V$ is within $(1 + \tilde\varepsilon)^2 = 1 +
\varepsilon$
of $\mathbb{E}[\bar W]/\mathbb{E}[\bar V] = Z(\beta)/Z(0)$,
completing the proof.
\end{pf*}

\subsection{The running time of the basic algorithm}\label{SBSrunningtime}
How many samples does Algorithm~\ref{ALGppaa} take on average?

\begin{teo}
\label{THMrunningtime}
When $n \geq4$, and $\varepsilon\leq1/10$,
Algorithm~\ref{ALGppaa} takes on average at most
\[
(q + 1)\bigl[5 + \bigl(2 + \ln(2n)\bigr) \bigl(14.9 \ln\bigl(100\bigl(2 +
\ln(2n)\bigr) (q+1)\bigr) + 48.2 \varepsilon^{-2}\bigr)\bigr]
\]
samples. For fixed $\varepsilon$
the number of samples is $O(q[\ln(n)(\ln(q) + \ln(\ln(n)))])$.
\end{teo}

\begin{pf}
A run of \textsf{TPA} uses a number of samples that is one plus a Poisson
random variable with mean $z(\beta) - z(0)$, so on average $q + 1$ samples.
So step 1 takes $5q + 5$ samples on average. From the concavity of
the $\ln$ function and Jensen's inequality, the second step takes at
most
\[
\bigl\lceil(2/3) \bigl(2 + \ln(2n)\bigr)\bigr\rceil\bigl\lceil22\ln\bigl(100\bigl(2 + \ln(2n)
\bigr) (q + 1)\bigr) \bigr\rceil q
\]
samples on average.
This is bounded above by
\[
q\bigl[14.9 \bigl(2 + \ln(2n)\bigr) \ln\bigl(100\bigl(2 + \ln(2n)\bigr) (q+1)
\bigr)\bigr].
\]
The resulting schedule has on average
at most $q/(d/k) + 1 = (2/3)[2 + \ln(2n)] q + 1$
intervals in it, and so the third step
of the algorithm\vadjust{\goodbreak} generates a number of samples that (on average) is at
most
\[
(2e\sqrt{10}) (2/3) \bigl(2 + \ln(2n)\bigr) (q + 1) \bigl((1 +
\varepsilon)^{1/2} - 1\bigr)^{-2}.
\]
When $\varepsilon\leq1/10$, $(1 + \varepsilon)^{1/2} - 1 \geq\varepsilon/2.05$,
so the number of samples in this section can be bounded by
\[
48.2\bigl(2 + \ln(2n)\bigr) (q + 1)\varepsilon^{-2}.
\]\upqed
\end{pf}




\printaddresses

\end{document}